\documentclass[11pt,reqno]{amsart}
\usepackage{a4wide}
\usepackage{amsmath,amssymb,amsfonts,euscript,enumerate,amsthm}
\usepackage{color}

\usepackage[colorlinks=true, urlcolor=blue, linkcolor=blue, citecolor=magenta]{hyperref}

\usepackage{graphicx}
\usepackage{epstopdf}
\usepackage[]{cleveref}
\usepackage{mathtools}
\usepackage{esint}
\usepackage{mathrsfs}
\usepackage[notref, notcite, final ]{showkeys}
\usepackage[normalem]{ulem}
\theoremstyle{plain}
\newtheorem{theorem}{Theorem}[section]
\newtheorem{lemma}[theorem]{Lemma}

\newtheorem{corollary}[theorem]{Corollary}

\theoremstyle{definition}
\newtheorem{definition}[theorem]{Definition}

\theoremstyle{remark}
\newtheorem{remark}[theorem]{Remark}

\numberwithin{equation}{section}

\makeatletter
\@namedef{subjclassname@2020}{\textup{2020} Mathematics Subject Classification}
\makeatother

\newcommand{\R}{{\mathbb R}}

\newcommand{\cL}{{\mathcal L}}

\newcommand{\al}{\alpha}
\newcommand{\be}{\beta}
\newcommand{\ga}{\gamma}
\newcommand{\de}{\delta}
\newcommand{\e}{\varepsilon}

\newcommand{\la}{\lambda}

\newcommand{\vp}{\varphi}

\newcommand{\Si}{\Sigma}
\newcommand{\Om}{\Omega}
\newcommand{\La}{\Lambda}

\newcommand{\Ga}{\Gamma}
\newcommand{\bS}{\mathbb{S}}

\newcommand{\ti}{\times}

\newcommand{\pa}{\partial}
\newcommand{\su}{\subset}
\newcommand{\qu}{\quad}

\newcommand{\sm}{\setminus}
\newcommand{\ra}{\rightarrow}

\newcommand{\D}{\overline{\nabla}}

\newcommand{\conv}{\operatorname{conv}}

\newcommand{\ov}{\overline}
\newcommand{\fr}{\frac}

\newcommand{\inn}[2]{\left\langle {#1},{#2} \right\rangle}
\newcommand{\norm}[1]{\left\lVert#1\right\rVert}

\def\({\left(}
\def\){\right)}
\def\<{\left\langle}
\def\>{\right\rangle}

\newcommand{\og}{\overline{g}}
\newcommand{\oD}{\overline{\nabla }}
\newcommand{\bed}{\beta_\delta}
\newcommand{\op}{(\partial_t -\mathcal{L})}

\title[]{Gauss curvature flow with shrinking obstacle}

\author{Ki-Ahm Lee}
\address{Department of Mathematical Sciences and Research Institute of Mathematics, Seoul National University, Seoul 08826, Republic of Korea}
\email{kiahm@snu.ac.kr}

\author{Taehun Lee}
\address{School of Mathematics, Korea Institute for Advanced Study, Seoul 02455, Republic of Korea}
\email{taehun@kias.re.kr}


\begin{document}

\begin{abstract}	
We consider a flow by powers of Gauss curvature under the obstruction that the flow cannot penetrate a prescribed region, so called an obstacle. 
For all dimensions and positive powers, we prove the optimal curvature bounds of solutions and all time existence with its long time behavior. We also prove the $C^1$ regularity of free boundaries under a uniform thickness assumption.
\end{abstract}

\maketitle


\section{Introduction}

We study an evolution of hypersurfaces by powers of Gauss curvature under the restriction that the hypersurface cannot enter a prescribed region. The prescribed region is called an obstacle, and we assume that it shrinks slowly with time.

Recall that for a given number $\al>0$, a one-parameter family of immersions $X:M^n\ti [0,T) \ra \R^{n+1}$ defining complete convex hypersurfaces $\Si_t=X(M^n,t)$ is a solution of the \textit{$\al$-Gauss curvature flow} if $X$ satisfies
\begin{align*}
\fr{\pa }{\pa t} X(p,t) = K^\al(p,t) \vec{n}(p,t).
\end{align*}
Here $K(p,t)$ and $\vec{n}(p,t)$ are the Gauss curvature and the inward unit normal vector of $\Si_t$ at $X(p,t)$, respectively. 

The classical Gauss curvature flow ($\al=1$) was first introduced by Firey \cite{Firey74_M} to model the process of wearing stone on a beach. 
Later, Chow \cite{Chow85_JDG} generalizes this flow to the $\al$-Gauss curvature flow ($\al>0$). 
Since then, many authors have studied this flow, and most of them focused on the analysis of singularities which was completely carried out when $\al \ge \tfrac{1}{n+2}$, see \cite{Andrews96_JDG,Andrews99_Invent,BCD17_Acta,AGN16_AM}
and the references therein. See also \cite{Andrews03_JAMS} for the case $0<\al<\tfrac{1}{n+2}$ and $n=1$. 
Note that the evolution of any closed initial hypersurface under the $\al$-Gauss curvature flow develops a singularity. 

We consider an $\al$-Gauss curvature flow with an obstacle in $\R^{n+1}$, which will block the development of collapsing of the hypersurface to a point. 
Precisely, given a strictly convex, closed hypersurface $\Si_0$ in $\R^{n+1}$, we consider a time dependent obstacle $\Phi=\Phi_t$ inside of $\Si_0$, which is a one-parameter family of strictly convex, closed hypersurfaces. We then evolve the hypersurface $\Si_0$ by the $\alpha$-Gauss curvature flow, $\al>0$, on the condition that the evolving hypersurface $\Si_t$ cannot penetrate the obstacle $\Phi_t$ for all time. Here the obstacle shrinks slowly with time, and for example one can consider obstacles of the form 
\begin{align*}
\Phi_t = \tfrac{1+e^{-t}}{2}\Phi_0.
\end{align*}
We remark that if the obstacle shrinks quickly so that it disappears in a finite time, then it cannot prevent the development of the collapse.
The precise definition of obstacles will be given in Definition \ref{def-shrinking} below.

Note that in general the evolutions of such hypersurfaces cannot be defined in the classical sense since the speed $\tfrac{\pa}{\pa t} X$ has jump discontinuities when the hypersurface $\Si_t$ touches the obstacle $\Phi_t$ for the first time. Thus we need to consider a generalized concept of solutions, that is, viscosity solutions.  

 A one-parameter family of immersions $X:M^n \ti [0,T) \ra \R^{n+1}$ with $\Si_t=X(M^n,t)$ is said to be a viscosity solution to the \textit{$\al$-Gauss curvature flow with a shrinking obstacle} if it satisfies
\begin{equation}\label{eq-main-X}
\begin{split}
\fr{\pa }{\pa t} X(p,t)&= K^\al(p,t)\vec{n}(p,t) \qu \text{for all} \qu X(p,t)\notin \Phi_t,
\\
\mathrm{conv}(\Si_t) &\supseteq \mathrm{conv}(\Phi_t) \qu \text{for all} \qu 0\le t< T,
\\
\inn{\fr{\pa }{\pa t} X}{\vec{n}}&\le  K^\al \qu \text{in} \qu M^n\ti [0,T),
\end{split}
\end{equation}
in the viscosity sense (see Definition \ref{def:vis-sol}), where $\mathrm{conv(\Si)}$ denotes the convex hull of $\Si$. 

As we discussed above, the solutions to \eqref{eq-main-X} are at most Lipschitz continuous in the time variable, which induces that solutions to \eqref{eq-main-X} have at most $C^{1,1}$ regularity. We will show that solutions to \eqref{eq-main-X} have indeed $C^{1,1}$ regularity, and therefore this is the optimal regularity. 

\begin{theorem}\label{t-main}
Let $\Si_0$ be a closed strictly convex smooth hypersurface that encloses a shrinking obstacle $\Phi$ (Definition \ref{def-shrinking}). Then \eqref{eq-main-X} starting from $\Si_0$ has a unique viscosity solution $\{\Si_t\}_{t\ge0}$ that exists for all time, and the solution $\{\Si_t\}_{t\ge0}$ has the optimal regularity which is $C^{1,1}(\bS^n\ti[0,\infty))$. In particular, principal curvatures $\la_1,\ldots,\la_n$ of $\Si_t$ are globally bounded away from zero and infinity. Moreover, there exists a finite time $T^*=T^*(n,\al,\Si_0,\Phi)$  such that $\Si_t=\Phi_t$ for $t\ge T^*$.
\end{theorem}

It is worth noting that our result for the optimal regularity relies on several curvature bounds. We first prove that the several curvatures are bounded by some constant that depends on $T$. By showing that the final shape of the solution becomes exactly the same shape as the obstacle after a finite time $T^*$, we can remove the time dependence in the curvature estimates.

To prove curvature estimates, we employ the method of penalization, which allows penetration of the obstacle up to order of $\delta>0$ and then recovers the original obstacle problems by pushing the hypersurfaces out of the obstacle as $\delta\rightarrow 0$.
In the other words, we obtain smooth approximate solutions, $u_{\delta}$, of \eqref{astde}  that will converge to the solution of \eqref{eq-main-X} as the penalty increases, i.e. $\delta\rightarrow 0$ .

A crucial ingredient in this paper is the smallest curvature estimate (Lemma \ref{lem-lac}) established by applying the maximum principle argument to the quantity
\begin{align*}
\lambda_{\min}^{-1}e^{-\chi\bed(u-\vp)},
\end{align*}
where $\la_{\min}=\min\{\la_1,\ldots,\la_n\}$ and $\bed$ is the penalty term defined in Section \ref{sec:pre}. The constant $\chi$ will be chosen in terms of the minimum value of Gauss curvature which we will estimate in Corollary \ref{c-K>c}.

Let us review some related results for stationary obstacles, $\Phi_t \equiv\Phi_0$.
In \cite{LL21_CVPDE}, we studied the obstacle problem in dimension $n=2$ with $0<\al\le 1$. In that paper, we established the optimal curvature bounds for principal curvatures, $0\le \la_1,\la_2\le C$. We point out that for stationary obstacles the lowest principal curvature may become zero. 
For the mean curvature flow, Rupflin and Schn\"{u}rer \cite{RS20_ASNSPCS} studied an obstacle problem by considering graphical mean curvature flows in one dimension higher. They proved local $C^{1,1}$ regularity for solutions in the graphical setting, which implies the existence of solutions for all time in the original dimension. See also \cite{ACN12_AIHPANL,GTZ19_GF,MN15_IFB}.

\bigskip

Once the optimal regularity of solutions to obstacle problems is established, the main interest will be a question of regularity of its free boundaries. In obstacle problems, a \textit{free boundary} of a solution is the interface between a \textit{coincidence set} and \textit{non-coincidence set} of the solution and its obstacle.

To find a local regularity of free boundaries, we use the following local graph representation of \eqref{eq-main-X}: for functions $w:Q_1\ra \R$ and $\phi:Q_1\ra \R$ such that $\Si_t$ and $\Phi_t$ are locally represented by the graphs of $w$ and $\phi$, respectively, we consider
\begin{equation}\label{eq-main-w}
\begin{alignedat}{3}
\fr{\pa}{\pa t} w &= \fr{\det(D^2w)^\al}{(1+|Dw|^2)^{\frac{(n+2)\al-1}{2}}} \qu &&\text{in} \qu \Om(w),
\\
w &\le \phi \qu &&\text{in} \qu Q_1
\\
\fr{\pa}{\pa t} w &\le \fr{\det(D^2w)^\al}{(1+|Dw|^2)^{\frac{(n+2)\al-1}{2}}} \qu &&\text{in} \qu Q_1, 
\end{alignedat}
\end{equation}
where $Q_r=B_r\ti (-r^2,0)\su \R^n\ti\R$ is the parabolic ball of radius $r$ and $\Om(w)=\{(x,t)\in Q_1: w(x,t)<\phi(x,t)\}$ is the non-coincidence set. Note that the coincidence set $\La(w)$ and the free boundary $\Ga(w)$ is then given by $\La(w)=Q_1\sm \Om(w)$ and $\Ga(w)=\pa \Om(w)$.

The regularity of free boundaries has been studied by many authors since the seminal work of Caffarelli \cite{Caffarelli77_Acta} which proves local $C^1$ regularity of free boundaries in obstacle problems for a class of elliptic equations and the Stefan problem. In the first named author's thesis \cite{Lee98_Thesis} and \cite{LS01_CPAM}, the authors analyzed the regularity of free boundaries in the obstacle problem for a fully nonlinear operator. See also \cite{LP21_MA,LPS19_CVPDE} for double obstacle problems. 

In the recent works \cite{FS14_ARMA,FS15_AMPA}, Figalli and Shahgholian proved $C^1$ regularity of free boundaries for general fully nonlinear elliptic/parabolic equations of the form $F(D^2w)=1$ or $F(D^2w)-\pa_tw=1$. Soon after Indrei and Minne \cite{IM16_AIHPANL} extended the result to the equations $F(D^2w,x)=f(x)$ and $F(D^2w,x,t)-\pa_tw= f(x,t)$, where $F$ and $f$ are Lipschitz continuous in both the space and time variable. We note that in these results $F$ is assumed to be convex in the hessian variable. 

However, our equation \eqref{eq-main-w} does not fall into the known cases. This is because our $F$ has $Dw$ dependence which is not Lipschitz continuous in the time variable although it is Lipschitz in the space variable after establishing the optimal $C^{1,1}$ regularity. We stress that most geometric PDEs have similar gradient dependence on the operator, which requires some extra work.

As indicated above, the theory for the free boundary regularity has been developed for convex operators. Hence we assume $\al \le \frac{1}{n}$ to make the operator $-\det(\cdot)^\al$ convex. 
We will prove $C^1$ regularity of the free boundary $\Ga(w)$ under the so called \textit{uniform thickness assumption} on the coincidence set $\La(w)$. 

To describe the thickness assumption, we first define the minimal diameter (also known as the minimal width) $\mathrm{MD(E)}$ of a set $E$ by the minimum distance between two parallel hyperplanes that contain the set $E$. Then the thickness of $\La(w)$ at $x_0$ in each time slice $t\in [t_0-r^2,t_0+r^2]$ is measured as
\begin{align*}
\mathfrak{d}_r(w,x_0,t_0) = \inf_{|t-t_0|\le r^2}\fr{\mathrm{MD}(\La(w)^t\cap B_r(x_0))}{r},
\end{align*}
where $\La(w)^t=\{x\in B_1:(x,t)\in \La(w)\}$ is the time slice of $\La(w)$ with respect to $t$.
We say that $\La(w)$ satisfies a uniform thickness condition if there exists $\e>0$ such that for all $0<r<\tfrac{1}{4}$ and $(x,t)\in\Ga(w)\cap Q_r$, it holds $\mathfrak{d}_r(w,x,t) > \e$.

\begin{theorem}\label{t-main-2}
Let $w$ be a solution to \eqref{eq-main-w} with $\al \le \frac{1}{n}$. If $\La(w)$ satisfies a uniform thickness condition, then the free boundary $\Ga(w)$ is locally a $C^1$ graph in space-time.
\end{theorem}

The thickness condition is necessary for Theorem \ref{t-main-2} since the free boundary $\Ga(w)$ may have arbitrary lower dimensional shapes such as a point when solutions to \eqref{eq-main-w} touch the obstacle for the first time.

\bigskip

We remark that there is another free boundary problem arising in the Gauss curvature flow, the Gauss curvature flow with flat sides. Hamilton \cite{Hamilton94_incollection} observed that if an initial hypersurface has flat sides, then its evolution by the Gauss curvature flow also has flat sides for some time, in contrast with the mean curvature flow whose flat sides disappear instantly. In the problem, free boundaries are the interface between flat sides and non-flat sides. Comparing with our obstacle problem that has growing contact sets, this free boundary problem has shrinking flat sides. In dimension $n=2$, regularity for the free boundary has been established by Daskalopoulos and Hamilton \cite{DH99_JRAM} for short time and by Daskalopoulos and the first named author \cite{DL04_Invent} for all time. See also \cite{CEI99_IUMJ} for waiting time effects when an initial data is smooth enough. Note that the same phenomena were observed for the $\al$-Gauss curvature flow when $\fr{1}{2}<\al\le 1$ and $n=2$, see \cite{KLR13_JDE}.

The paper is organized as follows. In Section \ref{sec:pre}, we fix geometric notation and define a shrinking obstacle. We also introduce a penalty term to approximate \eqref{eq-main-X} and present short time existence and evolution equations for the approximate solutions.
In Section \ref{sec:penalty}, we prove that the penalty term is bounded, which ensures uniform curvature estimates later and then the convergence of approximating solutions to solutions of the obstacle problem. In Section \ref{sec:speed>0}-\ref{sec:c<la<C}, we establish several curvature bounds possibly depending on the existing time $T$. By estimating the speed function, we prove upper and lower bounds for the Gauss curvature in Section \ref{sec:speed>0} and \ref{sec:speed<C}, respectively. We then obtain both the upper and lower bounds for principal curvatures in Section \ref{sec:c<la<C}. The proof of Theorem \ref{t-main} and \ref{t-main-2} are presented respectively in Section \ref{sec:pfThm1} and \ref{sec:pfThm2}.

\section{Preliminaries}\label{sec:pre}
Throughout the paper we use the Einstein summation convention in which repeated upper and lower indices are summed. 

For a smooth hypersurface $\Si^n$ in $\R^{n+1}$, we denote the induced metric by $g=\{g_{ij}\}$, the second fundamental form by $h=\{h_{ij}\}$, and the Weingarten map by $W=\{h_i^j\}$. The principal curvatures are the eigenvalues of $W$ denoted by $\la_1\le \cdots \le \la_n$. We will use the mean curvature $H=\la_1+\cdots+\la_n$ and the Gauss curvature $K=\la_1\cdots \la_n$.

If $\Si^n$ is compact without boundary and strictly convex, then the Gauss map $\nu:\Si^n\ra \bS^n$ given by the outward unit normal vector to $\Si^n$ is a diffeomorphism. In the case we describe the hypersurface $\Si^n$ as its support function $u:\bS^n \ra \R$ defined by $u(z)=\inn{\nu^{-1}(z)}{z}$. In other words, the hypersurface $\Si^n$ is the image of the embedding $\nu^{-1}(z)=\oD  u(z) +u(z) z$, where $\oD$ is the Levi-Civita connection on $\bS^n$ induced by the standard metric $\og=\{\og_{ij}\}$. It can be checked that the second fundamental form $h$ is expressed as
\begin{align}\label{eq-hu}
h_{ij} = \oD_i\oD_j u + u \og_{ij},
\end{align}
and that the eigenvalues of $h_{ij}$ with respect to $\og$ are the reciprocal of the principal curvatures, i.e., 
\begin{align}\label{eq-hggb}
h_{ij}\og^{jk} = g_{ij}b^{jk}
\end{align}
where $b=\{b^{ij}\}$ is the inverse of the second fundamental form $h$. 

\,\\
\noindent
\textbf{Initial data.}
Let $\Si_0$ be a strictly convex, closed initial hypersurface embedded in $\R^{n+1}$ and $u_0:\bS^n\ra \R$ be the support function of $\Si_0$. Both $\Si_0$ and $u_0$ are used as initial data. Geometric quantities associated to the initial data will be used with subscript $0$, for example $H_0$, $K_0$, and so on.

\,\\
\noindent
\textbf{Obstacle.}
Let $\{\Om(t)\}_{t\ge 0}$ be a family of bounded open sets that are strictly convex. We always assume that $\Om(t)$ shrinks in $t$, that is, $\Om(t_1)\supset\Om(t_2)$ for $t_1<t_2$, and that the limit $\Om_{\infty}:=\cap_{t\ge 0} \Om(t)$ has an interior point.
For convenience, we assume that the limit obstacle $\Om_\infty$ contains the origin.
We also assume that the initial obstacle $\Om(0)$ is strictly enclosed by the initial data, i.e., $\ov{\Om(0)}\subset \operatorname{conv}(\Si_0)$, where $\operatorname{conv}(\Si_0)$ denotes the convex hull of $\Si_0$.
We denote the boundary of $\Om(t)$ by $\Phi_t$ and the support function of $\Phi_t$ by $\vp(\cdot, t):\bS^n\ra \R$. Denote furthermore the principal curvatures of $\Phi_t$ by $\mu_1,\cdots,\mu_n$ (not necessarily in increasing order).
Other geometric quantities associated to the obstacle will be used with subscript $\Phi$, for example $H_\Phi$, $K_\Phi$, and so on.

\begin{definition}[shrinking obstacles]\label{def-shrinking}
An obstacle $\Phi$ is said to be \textit{shrinking} if its support function $\vp$ is in $C^{3,1}(\bS^n\ti[0,\infty))$ and satisfies (i) $\pa_t\vp<0$, (ii) the speed $-\pa_{t}\vp$ is non-increasing, (iii) the final shape $\operatorname{conv}(\Phi_\infty)$ has an interior point, and (iv) the principal curvatures $\mu_1,\ldots,\mu_n$ are non-decreasing and bounded.
\end{definition}

A natural example of the shrinking obstacle is any strictly convex hypersurface that homothetically shrinks with decreasing speed, i.e., if $A:[0,\infty)\ra (0,1]$ with $A(0)=1$ and $\lim_{t\ra \infty}A(t)>0$ is decreasing and convex, then $A(t)\vp_0(z)$ is a shrinking obstacle. More generally, given an initial shape $\Om_0$ and a final shape $\Om_\infty$ with their support function $\vp_0$ and $\vp_\infty$ satisfying $\vp_0>\vp_\infty>0$, the obstacle defined by
\begin{align*}
\vp(\cdot,t) = e^{-t} \vp_0+ (1-e^{-t})\vp_\infty
\end{align*}
is clearly shrinking.

Note that for $a>0$, a rescaled obstacle $\vp^a(\cdot,t)=\vp(\cdot,at)$ is also shrinking and has the same initial and final shape with $\vp$. Moreover, given an initial data $\Si_0$, the following holds if we choose $a$ small enough:
\begin{align*}
-\pa_t \vp^a(\cdot,0)=-a\pa_t\vp_0 < \min_{\bS^n}K_0^\al \qu \text{and} \qu -\pa_t \vp^a(\cdot,0)<\min_{\bS^n}K_{\Phi_0}^\al.
\end{align*}
Thus we additionally assume for simplicity that the shrinking obstacle satisfies
\begin{align*}
-\pa_t \vp_0 < \min_{\bS^n} K_0^\al \qu \text{and} \qu -\pa_t\vp_0<\min_{\bS^n} K_{\Phi_0}^\al.
\end{align*}

\begin{remark}
The condition (ii) and (iv) in Definition \ref{def-shrinking} ensure that the obstacle is a supersolution of the $\al$-Gauss curvature flow. Indeed, since the principal curvatures $\mu_i$ of $\Phi(\cdot,t)$ are non-decreasing and the speed $-\pa_t\vp$ is non-increasing, we obtain that $-\pa_t \vp\le-\pa_t \vp_0\le K_{\Phi_0}^\al \le K_{\Phi_t}^\al$.
We point out that if the obstacle $\Phi$ is not a supersolution, then the solution to \eqref{eq-main-X} might be separated from the obstacle after it has the same shape with the obstacle. 
\end{remark}

Let $M^n$ be an $n$-dimensional smooth manifold, and let $X(\cdot,t):M^n\ra \R^{n+1}$ ($0\le t< T$) be a one-parameter family of smooth immersions for some $T\in (0,\infty]$ with an image $\Si_t=X(M^n,t)$ which is a strictly convex, closed hypersurface. For the family $\{\Si_t\}$, we still denote by $g=\{g_{ij}\}$ the induced metrics and by $h=\{h_{ij}\}$ the second fundamental forms, where they depend on the variable $t$. We say that $u:\bS^n\ti [0,T)\ra \R$ is a support function of the family $\{\Si_t\}$ if $u(\cdot,t):\bS^n\ra \R$ is the support function of $\Si_t$.

In terms of the support function, we rewrite the obstacle problem \eqref{eq-main-X} as
\begin{equation}\label{eq-main-u}
\begin{alignedat}{3}
-\pa_t u &= K^\al \qu &&\text{in} \qu \{u>\vp\}, \\
u &\ge \vp \qu &&\text{in}\qu \bS^n\ti[0,T),\\
- \pa_t u &\le K^\al \qu &&\text{in}\qu\bS^n\ti[0,T).\\
\end{alignedat}
\end{equation}
 Here $u$ and $\vp$ are the support functions of $\Si_t$ and $\Phi_t$, respectively, and $K$ is the Gauss curvature of $\Si_t$.
 Moreover, by \eqref{eq-hu} and \eqref{eq-hggb}, the Gauss curvature becomes 
\begin{align}\label{K=det}
K=\det (g^{ij}h_{jk})=\det (b^{ij}\og_{jk})=\fr{\det (\og_{ij})}{\det (h_{ij})}=\fr{\det (\og_{ij})}{\det (\oD_i\oD_ju+u\og_{ij})}.
\end{align}
Combining \eqref{eq-main-u} and \eqref{K=det}, we rewrite the obstacle problem \eqref{eq-main-X} again as 
\begin{align}\label{eq-main-min}
\min\left\{\pa_tu+\left(\fr{\det (\og_{ij})}{\det (\oD_i\oD_ju+u\og_{ij})}\right)^\al,u-\vp\right\}=0.
\end{align}

We recall the notion of viscosity solutions to \eqref{eq-main-min} or equivalently \eqref{eq-main-X}.
\begin{definition}[viscosity solution]\label{def:vis-sol}
A continuous function $u\in C(\bS^n\ti [0,T))$ is a \textit{viscosity subsolution (supersolution)} of \eqref{eq-main-min} if for any point $(x_0,t_0)\in \bS^n\ti(0,T)$, the left hand side of \eqref{eq-main-min} is nonpositive (nonnegative) for all test functions $\eta\in C^2(\bS^n\ti[0,T))$ touching $u$ from above (below) at $(x_0,t_0)$, i.e., $\eta(x_0,t_0)=u(x_0,t_0)$ and $\eta(x,t)\ge (\le) u(x,t)$. We say that $u$ is a \textit{viscosity solution} of \eqref{eq-main-min} if it is both a subsolution and supersolution of \eqref{eq-main-min}. 
\end{definition}

Let $\be: \R \ra \R$ be a non-decreasing, concave, smooth function such that
\begin{align*}
\be(x)=0 \text{ for } x\ge 1,\qu \be''(x)=0 \text{ for } x<0, \qu \be(0)=-1.
\end{align*}
Given a constant $\de>0$ and an shrinking obstacle $\Phi$ in Definition \ref{def-shrinking}, we define $\bed(x) = C_0 \be(x/\de)$, where $C_0=\norm{K_\Phi}_{L^\infty(\bS^n\ti [0,\infty))}^\al<\infty$. For convenience, we choose $\de<\min_{\bS^n}(u_0-\vp_0)$ so that $\bed(u-\vp)=0$ at the initial time. 

To prove the existence and regularity results, we approximate \eqref{eq-main-u} by the following singular perturbation problem: 
\begin{equation}\label{astde}\tag{$\ast_\de$}
\begin{alignedat}{3}
-\pa_t u &=K^\al+\bed(u-\vp) \qu &&\text{in}\qu \bS^n\ti (0,T),
\\
u(\cdot,0)&= u_0 \qu &&\text{on} \qu \bS^n.
\end{alignedat}
\end{equation}
We say that a one-parameter family $\{\Si_t\}_{0\le t\le T}$ of hypersurfaces $\Si_t$ is a solution to \eqref{astde} if the support function of $\Si_t$ satisfies \eqref{astde}. We write $\bed$ instead of $\bed(u-\vp)$.

Let $\cL$ denote the highest order terms of the linearized operator of $-K^\al$. Then it follows from \eqref{K=det} that
\begin{align*}
\cL=\al K^\al b^{ij}\oD_i\oD_j.
\end{align*} 
The associated inner product $\inn{\cdot}{\cdot}_\cL$ and norm $\norm{\cdot}_\cL$ are defined by $\inn{\oD A}{\oD B}_\cL= \al K^\al b^{ij} \oD_iA\oD_jB$ and $\norm{\oD A}_\cL^2 = \inn{\oD A}{\oD A}_\cL$, respectively, for smooth functions $A$ and $B$ on $\bS^n\ti [0,T)$.

The short time existence follows from the standard argument as in \cite{Chow85_JDG} (see also \cite{LL21_CVPDE} for instance) that is based on the inverse function theorem. In fact, since the linearized operator of $\bed(u-\vp)$ does not produce any second order term, the proof in \cite{Chow85_JDG} can be applied without modifications.

\begin{lemma}[Short time existence]\label{lem-locE}
If $\Si_0=X_0(M^n)$ is a strictly convex hypersurface, then \eqref{astde} has a unique smooth solution $\{\Si_t\}_{0\le t< T}$ for short time. 
\end{lemma}

Under the flow \eqref{astde}, geometric quantities also evolve. 

\begin{lemma}\label{evs}
Let $\{\Si_t\}$ be a solution to \eqref{astde} and $u:\bS^n\ti(0,T]\ra \R$ be the support function of $\{\Si_t\}$. Then the following holds:
\begin{alignat}{3}
\label{ev-u} & \op u&&=\al K^\al Hu-(n\al +1)K^\al -\bed\\
\label{ev-Ka} & \op K^\al&&= \al K^{2\al} H +\cL \bed+\al K^\al H \bed\\
\label{ev-h} & \op h_{ij}&&=-\mathcal{C}_{ij}+(n\al-1)K^\al \og_{ij}-\al K^\al H h_{ij} -\oD^2_{ij}\bed -\bed \og_{ij}
\end{alignat}
where $\mathcal{C}_{ij}= \al^2 K^\al b^{kl}b^{pq}\oD_ih_{pq}\oD_jh_{kl}+ \al K^\al b^{kp}b^{lq}\oD_ih_{pq}\oD_jh_{kl}$.
\end{lemma}

\begin{proof}
By \eqref{eq-hu} and \eqref{eq-hggb}, we have
\begin{align*}
\cL u= \al K^\al b^{ij} \D_i\D_j u=\al K^\al b^{ij}(h_{ij}-u\og_{ij})=n\al K^\al -\al K^\al  Hu.
\end{align*}
Then the equation \eqref{ev-u} follows from $\pa_t u = -K^\al -\bed$.

To prove \eqref{ev-Ka}, we recall \eqref{K=det} so that 
\begin{align*}
\pa_t K = -Kb^{ij}\pa_t h_{ij}=Kb^{ij} (\D_i\D_j (-\pa_t u)+(-\pa_t u)\og_{ij}).
\end{align*}
Using $-\pa_t u = K^\al +\bed$ and \eqref{eq-hggb} again, we see that 
\begin{align*}
\pa_t K^\al =\al K^{\al-1}\pa_t K=\cL(K^\al +\bed)+ (K^\al +\bed)\al K^\al H
\end{align*}
which gives \eqref{ev-Ka}.

For the last assertion, note first that 
\begin{align}\label{eq-ht}
\pa_t h_{ij}=-\D_i\D_j (K^\al+\bed) -(K^\al +\bed) \og_{ij}.
\end{align} On the other hand, since $\D_j K^\al=-\al K^\al b^{kl}\D_jh_{kl}$ and $\D_i b^{kl}=-b^{kp}b^{lq}\D_ih_{pq}$, we have

\begin{equation}
\begin{split}\label{eq-Kaij}
\D_i\D_j K^\al &= \al^2 K^\al b^{kl}b^{pq}\D_ih_{pq}\D_j h_{kl}+\al K^\al b^{kp}b^{lq}\D_ih_{pq}\D_jh_{kl}-\al K^\al b^{kl}\D_i\D_j h_{kl}
\\
&= \mathcal{C}-\al K^\al b^{kl} \D_i\D_j h_{kl}.
\end{split}
\end{equation}
To proceed further we need the following result that we will prove in the next lemma: 
\begin{align*}
\oD_k\oD_l h_{ij}=\oD_i\oD_jh_{kl}-h_{kj}\og_{il}+\og_{kl}h_{ij}-h_{kl}\og_{ij}+\og_{kj}h_{il}.
\end{align*}
Multiplying with $\al K^\al b^{kl}$ we get
\begin{align*}
\cL h_{ij}= \al K^\al b^{kl}\D_i\D_jh_{kl}+\al K^\al(-\og_{ij}+Hh_{ij}-n\og_{ij}+\og_{ij}),
\end{align*}
and substituting this into \eqref{eq-Kaij} we obtain 
\begin{align*}
\D_i\D_j K^\al =\mathcal{C}_{ij}-\cL h_{ij} +\al K^\al (Hh_{ij}-n\og_{ij}).
\end{align*}
This together with \eqref{eq-ht} gives the desired result.
\end{proof}

\begin{lemma}\label{l-codazzi}
The second fundamental form $h_{ij}$ satisfies 
\begin{enumerate}[(i)]
\item\label{e-codazzi} $\oD h$ is totally symmetric, i.e., $\oD_k h_{ij}=\oD_ih_{kj}=\oD_jh_{ki}$,
\item $\oD_k\oD_l h_{ij}=\oD_i\oD_jh_{kl}-h_{kj}\og_{il}+\og_{kl}h_{ij}-h_{kl}\og_{ij}+\og_{kj}h_{il}$.
\end{enumerate}
\end{lemma}

\begin{proof}
Since $\oD_k h_{ij}= \oD_k h_{ji}$, it suffices to show that $\oD_kh_{ij}=\oD_ih_{kj}$. Recall that $h_{ij}= \oD_i\oD_j u+u\ov g_{ij}$ and $\oD \ov g=0$. Then we compute
\begin{align*}
\oD_k h_{ij} &=\oD_k (\oD_i\oD_j u +u\ov g_{ij})= \oD_i\oD_k\oD_ju-\ov R_{kipj}\oD^pu+(\oD_ku)\ov g_{ij},
\end{align*}
where $\ov R_{kipj}$ denotes the Riemannian curvature tensor on the sphere. 
Since 
\begin{align}\label{e-Rh}
\ov R_{ijkl}=\og_{ik}\og_{jl}-\og_{il}\og_{jk},
\end{align}
we find 
\begin{align*}
\oD_k h_{ij}=\oD_i(h_{kj}-u\ov g_{kj})-(\ov g_{kp}\ov g_{ij}-\ov g_{kj}\ov g_{ip})\oD^pu+(\oD_ku)\ov g_{ij}=\oD_ih_{kj},
\end{align*}
and the first assertion follows.

To prove the second assertion, note that we have shown $\oD_lh_{ij}=\oD_ih_{lj}$. Again by \eqref{e-Rh}, we obtain that
\begin{align*}
\oD_k\oD_lh_{ij}&=\oD_k\oD_ih_{lj}=\oD_i\oD_kh_{lj}-\ov R_{kipl} \og^{pr}h_{rj}-\ov R_{kipj} \og^{pr}h_{rl}\\
&=\oD_i\oD_jh_{kl}-(\og_{kp}\og_{il}-\og_{kl}\og_{ip}) \og^{pr}h_{rj}-(\og_{kp}\og_{ij}-\og_{kj}\og_{ip}) \og^{pr}h_{rl}\\
&=\oD_i\oD_jh_{kl}-h_{kj}\og_{il}+\og_{kl}h_{ij}-h_{kl}\og_{ij}+\og_{kj}h_{il}
\end{align*}
and the conclusion follows.
\end{proof}

\section{Uniform boundedness of the penalty term}\label{sec:penalty}
To extract a solution from approximate solutions of \eqref{astde}, we need several estimates that will be presented in this and the following three sections. In this section we establish the uniform boundedness of the penalty term $\bed(u-\vp)$. Once we obtain the estimate, any possible limit of subsequence of the approximate solutions can not pass through the obstacle $\Phi$. Indeed,
since $\bed(x)\ra -\infty$ as $\de\ra 0$ for each $x<0$, we conclude from the uniform boundedness of $\bed(u-\vp)$ that
\begin{align*}
\liminf_{\de\ra0} (u-\vp)\ge 0.
\end{align*}

\begin{lemma}\label{bed-bdd}
Let $u$ be a solution of \eqref{astde} in $\bS^n\ti [0,T)$. Then 
\begin{align*}
-C_0\le \bed(u(z,t)-\vp(z,t)) \le 0 \qu \text{for} \qu (z,t)\in \bS^n\ti[0,T),
\end{align*}
where $C_0=\norm{K_\Phi}_{L^\infty(\bS^n\ti [0,\infty))}^\al$ is independent of $\de$.
\end{lemma}

\begin{proof}
The inequality $\bed\le0$ follows from the definition of $\bed$. To obtain the lower bound, assume that there exists a point $(z_0,t_0)$ in $\bS^n\ti [0,T)$ such that $u=\vp$ at $(z_0,t_0)$ for the first time. Since the initial hypersurface $\Si_0$ strictly encloses the initial obstacle $\Phi_0$, we see $t_0>0$, and $u-\vp$ attains an interior minimum at $(z_0,t_0)$ over $\bS^n\ti [0,t_0]$. At the minimum point, we have
\begin{align*}
\oD(u-\vp)=0,\qu \oD^2(u-\vp)\ge0, \qu (u-\vp)_t\le 0, \qu u=\vp,
\end{align*}
and it follows 
\begin{align*}
K=\fr{\det \og_{ij}}{\det(\oD_i\oD_j u+u \og_{ij})}\le \fr{\det \og_{ij}}{\det(\oD_i\oD_j \vp+\vp \og_{ij})}=K_{\Phi}.
\end{align*}
Then 
\begin{align*}
-\vp_t\le -u_t= K^\al +\bed(u-\vp)\le K_\Phi^\al+\bed(0),
\end{align*}
and hence it is contradict to our choice of $\bed(0)=-\norm{K_\Phi}_{L^\infty(\bS^n\ti [0,\infty))}^\al$ in the definition of the penalty term since $\vp_t <0$.
Therefore, we conclude $u>\vp$ for all $(z,t)\in \bS^n\ti [0,T)$. Hence, the monotonicity of the function $\bed:\R \ra \R$ implies
$$\bed(u-\vp)\ge \bed(0)=-\norm{K_{\Phi}}_{L^\infty(\bS^n\ti [0,\infty))}^\al=-C_0$$
for all $(x,t)\in \bS^n\ti[0,T)$.
\end{proof}

\begin{remark}
In the proof, we have shown $u>\vp$ for all $(z,t)\in \bS^n\ti [0,T)$, which means that the evolving hypersurface $\Si_t$ under \eqref{astde} cannot penetrate or even touch the obstacle. This is because $|\bed(0)|$ is chosen sufficiently large. If one consider that the condition $|\bed(0)|$ has another uniform constant, for example $|\bed(0)|=1$, then one can prove $u>\vp-C(\de)$ for some constant $C(\de)$ with $C(0+)=0$, in which case the hypersurface might penetrate the obstacle but the depth of penetration is controlled.
\end{remark}

\section{Speed estimates I: Uniform positive lower bounds}\label{sec:speed>0}

In this section we prove uniform positive lower bounds for the speed of the solution to \eqref{astde}. As a direct corollary, we also obtain uniform positive lower bounds on the Gauss curvature of $\Si_t$.

\begin{lemma}\label{l-uvpt}
Let $\Si_0$ be an initial hypersurface and $\Phi$ be a shrinking obstacle. If $u:\bS^n\ti [0,T)\ra \R$ is a smooth solution to \eqref{astde}, then
\begin{align}\label{i-uvpt}
\pa_t(u-\vp)\le 0 \qu \text{in } \bS^n\ti [0,T).
\end{align}
\end{lemma}

\begin{proof}
Since $\Si_0$ is strictly convex and $\min_{\bS^n} (u_0-\vp(\cdot,0))\ge\de>0$ by our choice of $\de$, we have at $t=0$, $\pa_t u =-K^\al -\bed = -K_0^\al <0$. On the other hand, from the definition of the shrinking obstacle $\Phi$, we have 
\begin{align*}
-\pa _t \vp(\cdot, 0) <\min_{\bS^n}K_0^\al\le  K_0^\al
\end{align*}
which implies \eqref{i-uvpt} at $t=0$.

Recall that $K = \det \og_{ij} / \det (\oD_i\oD_j u +u\og_{ij})$. If we differentiate $-\pa_t u=K^\al +\bed$ with respect to $t$, we obtain
\begin{align}\label{ev-ut}
\pa_t(-\pa_t u)=-\al K^\al b^{ij}(\oD_i\oD_j \pa_t u+\pa_t u \og_{ij})+\bed' (u-\vp)_t.
\end{align}
For $Z(x,t):=\pa_t(u(x,t)-\vp(x,t))$, \eqref{ev-ut} can be rewritten as
\begin{align*}
-Z_t-\pa^2_t\vp= -\al K^\al b^{ij}(\oD_i\oD_j(Z+\pa_t\vp)+(Z+\pa_t\vp)\og_{ij})+\bed'Z.
\end{align*}
Thus we have
\begin{align*}
Z_t=\cL Z +(\al K^\al H-\bed') Z +\al K^\al b^{ij}(\oD_i\oD_j\pa_t\vp+\pa_t\vp\og_{ij})-\pa_t^2\vp.
\end{align*}

To estimate the terms involving the obstacle $\vp$, recall that the shrinking obstacle satisfies $\pa_t^2 \vp\ge0$ and $\pa_t\mu_i(\cdot,t)\ge0$, where $\mu_1(\cdot,t), \cdots, \mu_n(\cdot,t)$ is the principal curvatures of $\Phi_t$. Since the second fundamental form of the obstacle is given by
\begin{align*}
h_{ij}^{\vp}= \oD_i\oD_j\vp+\vp \og_{ij},
\end{align*}
it follows from the strict convexity of the solution and the property $\pa_t\mu_i\ge0$ that
\begin{align*}
\al K^\al b^{ij}(\oD_i\oD_j\pa_t\vp+\pa_t\vp \og_{ij})=\al K^\al b^{ij}\pa_t(h_{ij}^{\vp})=\al K^\al b^{ij}\og_{jk}\pa_t(\og^{kl}h_{lj}^\vp)\le0.
\end{align*}
Note that in the last inequality we have used that the eigenvalues of $\og^{kl}h_{lj}^\vp=(b^\vp)^{kl}g^\vp_{lj}$ are $1/\mu_1,\cdots, 1/\mu_n$ which are decrease in time. Hence, 
\begin{align}\label{i-Z}
Z_t\le\cL Z +(\al K^\al H-\bed') Z.
\end{align}

Assuming the contrary, we take a time $t_0$ such that $\sup_{\bS^n\ti[0,t_0]}Z>0$. Set $m=\sup_{\bS^n\ti[0,t_0]}(\al K^\al H-\bed')<\infty$ and let $\tilde Z = Ze^{-mt-t}$. By \eqref{i-Z}, $\tilde Z$ satisfies 
\begin{align}\label{i-tZ}
\tilde Z_t = Z_t e^{-mt-t}-(m+1)Ze^{-mt-t}\le \cL \tilde Z -\tilde Z.
\end{align}
For a small number $\e>0$, we take a point $(x_1,t_1)$ satisfying $\tilde Z(x_1,t_1)=\e$ for the first time. Clearly, $t_0>0$.
Then we have that $\tilde Z\le \e$ on $\bS^n\ti [0,t_1]$ and $\tilde Z(z_1,t_1)=\e$,
and that
\begin{align*}
\pa_t \tilde Z\ge0, \qu \oD_i\tilde Z=0, \qu \oD_i\oD_j \tilde Z \le 0 \qu \text{at } (z_1,t_1),
\end{align*}
which is contradict to \eqref{i-tZ}. Thus $\sup_{\bS^n\ti[0,t]}Z\le 0$ for all $t\in(0,T)$.
\end{proof}

Using the lemma above, we obtain the uniform positive lower bound for the Gauss curvature. 

\begin{corollary}\label{c-K>c}
Let $\Si_0$ be an initial hypersurface and $\Phi$ be a shrinking obstacle. If $\{\Si_t\}_{0\le t< T}$ is the solution to \eqref{astde}, then the Gauss curvature of $\Si_t$ has a uniform positive lower bound, i.e., 
\begin{align*}
\inf_{\bS^n\ti [0,T)}K\ge c_T>0,
\end{align*}
where $c_T=c(\al,T,\Phi)$ is a constant independent of $\de$.
\end{corollary}

\begin{proof}
By Lemma \ref{l-uvpt} and the non-positivity of $\bed$, we have 
\begin{align*}
K^\al = -\pa_t u -\bed \ge -\pa_t u \ge -\pa_t \vp\ge \min_{\bS^n\ti[0,T]}(-\pa_t\vp)>0.
\end{align*}
Now the conclusion follows by choosing $c_T=\min_{\bS^n\ti[0,T]}(-\pa_t\vp)^{1/\al}$.
\end{proof}

We will remove the time dependence of the constant $c_T$ above after analyzing the long time behavior of the solution $\{\Si_t\}$. 

\section{Speed estimates II: Uniform upper bounds}\label{sec:speed<C}
In the previous section we proved the positive lower bound on the speed $-\pa_t u$ and the Gauss curvature $K$. In this section we will obtain the opposite bounds on the speed and the Gauss curvature.

\begin{lemma}\label{lem-speed-upper}
Let $\Si_0$ be an initial hypersurface and $\Phi$ be a shrinking obstacle. If $u:\bS^n\ti [0,T)\ra \R$ is a smooth solution of the penalized problem \eqref{astde}, then 
\begin{align*}
-\pa_t u \le C \qu \text{in } \bS^n\ti [0,T),
\end{align*}
where $C=C(n,\al,\Si_0,\Phi)$ is a constant independent of $\de$.
\end{lemma}

\begin{proof}
Let $\vp$ be the support function of the obstacle and $\rho_0=\fr{1}{2}\min_{\bS^n} \vp_\infty>0$. We consider an auxiliary function on $\bS^n\ti[0,T)$
\begin{align*}
w= \fr{K^\al(z,t) +\bed(u(z,t)-\vp(z,t))}{u(z,t)-\rho_0}.
\end{align*}
Notice that the denominator remains positive since $u\ge \vp\ge \vp_\infty \ge 2\rho_0$ by Lemma \ref{bed-bdd} and the definition of $\rho_0$.

Our first task is to derive the evolution equation for the quantity $w$. Since 
\begin{align*}
\oD_j w&= \fr{\oD_j(K^\al +\bed)}{u-\rho_0}-\fr{K^\al+\bed}{(u-\rho_0)^2}\oD_j(u-\rho_0),\\
\cL w&= \fr{\cL(K^\al +\bed)}{u-\rho_0}-\fr{2\inn{\oD(K^\al+\bed)}{\oD u}}{(u-\rho_0)^2}-\fr{K^\al+\bed}{(u-\rho_0)^2}\cL u+\fr{2(K^\al+\bed)}{(u-\rho_0)^3}\norm{u}_\cL^2
\\
&=\fr{\cL(K^\al +\bed)}{u-\rho_0}-\fr{K^\al+\bed}{(u-\rho_0)^2}\cL u+\fr{2}{u-\rho_0}\inn{\oD w}{\oD u}_\cL,
\end{align*}
we obtain
\begin{align}
\begin{split}\label{eq-w1}
\op w&=\fr{\op(K^\al +\bed)}{u-\rho_0}-\fr{K^\al +\bed}{(u-\rho_0)^2}\op u-\fr{2}{u-\rho_0}\inn{\oD w}{\oD u}_\cL.
\end{split}\end{align}
It follows from \eqref{ev-Ka} in \Cref{evs} that
\begin{align*}
\op(K^\al +\bed)=\al K^{\al} H(K^\al+\bed) +  (u-\vp)_t\bed' .
\end{align*}
Plugging this and the evolution equation \eqref{ev-u} in \Cref{evs} into \eqref{eq-w1} gives that 
\begin{align*}
\op w&=\fr{  (u-\vp)_t\bed'}{u-\rho_0}-\fr{K^\al +\bed}{(u-\rho_0)^2}(\al K^\al H\rho_0-(n\al +1)K^\al -\bed)\\
&\qu-\fr{2}{u-\rho_0}\inn{\D w}{\D u}_\cL.
\end{align*}
Therefore, we arrive at 
\begin{align*}
\op w&=-\fr{2}{u-\rho_0}\inn{\D w}{\D u}_\cL\\
&\qu +\(-w-\fr{\vp_t}{u-\rho_0}\)\bed'-\fr{w}{u-\rho_0}(\al K^\al H\rho_0-(n\al +1)K^\al -\bed).
\end{align*}

Next, we apply the maximum principle argument to $w$. Fix a time $T'\in (0,T)$. Since $\bS^n\ti [0,T']$ is a compact set and $\Si_0$ is strictly convex, $w$ attains positive maximum value at some point $(z_0,t_0)\in \bS^n\ti [0,T']$. If $t_0=0$, then 
\begin{align}\label{ineq-wK}
w\le \fr{\max_{\bS^n} (K _{0})^\al}{\rho_0},
\end{align} 
where $K_{0}$ is the Gauss curvature of the initial hypersurface $\Si_0$.
Now we assume that $t_0>0$. Then $w$ has an interior maximum in $\bS^n\ti [0,T']$ and thus $w$ satisfies at the interior maximum point
\begin{align*}
\D w =0, \qu \D^2 w\le 0, \qu \pa_t w\ge 0.
\end{align*}
This gives that
\begin{align*}
0\le \(-w-\fr{\vp_t}{u-\rho_0}\)\bed'-\fr{w}{u-\rho_0}(\al K^\al H\rho_0-(n\al +1)K^\al -\bed)
\end{align*}
at $(z_0,t_0)$. If $\(-w-\fr{\vp_t}{u-\rho_0}\)(z_0,t_0)\ge0$, then at the same point
\begin{align}\label{ineq-wtau}
w \le \fr{-\pa_t\vp}{\rho_0}\le\fr{-\pa_t\vp_0}{\rho_0}<\fr{\min_{\bS^n} K_0^\al}{\rho_0}
\end{align}
since $\vp$ is convex in time variable.
Otherwise, it follows from $\bed\le 0$ and $\bed'\ge0$ that
\begin{align*}
0\le \fr{w}{u-\rho_0}(-\al K^\al H\rho_0+(n\al +1)K^\al )
\end{align*}
which implies by the arithmetic-geometric mean and $K^\al>0$,
\begin{align}\label{ineq-KHC}
nK^{\fr{1}{n}}\le H\le \fr{n\al+1}{\al\rho_0}
\end{align}
at the point $(z_0,t_0)$. Hence, combining \eqref{ineq-wK}, \eqref{ineq-wtau} and \eqref{ineq-KHC}, we conclude that
\begin{align*}
\max_{\bS^n\ti [0,T']}w\le \fr{1}{\rho_0}\max\left\{  (\max_{\bS^n} K_{0})^\al, \(\fr{n\al+1}{n\al\rho_0}\)^{n\al}\right\}=:C_1.
\end{align*}
Since $T'$ is an arbitrary number in $(0,T)$ and the constant $C_1$ does not depends on $T'$, we conclude that $\max_{\bS^n\ti [0,T)} w\le C_1$.

Finally, by Lemma \ref{l-uvpt}, we observe $u(z,t)\le u_0(z)\le \max_{\bS^n} u_0$, which completes the proof since $-\pa_t u \le w(u-\rho_0) \le C_1 \max_{\bS^n} u_0=:C$.
\end{proof}

Using the uniform boundedness of the penalty term $\bed(u-\vp)$ and upper bounds on the speed, we can obtain the following upper bound on the Gauss curvature $K$. 

\begin{lemma}\label{lem-KC}
Let $\Si_0$ be an initial hypersurface and $\Phi$ be a shrinking obstacle. If $\{\Si_t\}_{0\le t< T}$ be a solution to \eqref{astde}, then the Gauss curvature of $\Si_t$ satisfies
\begin{align*}
\max_{\bS^n\ti [0,T)}K\le C,
\end{align*}
where $C=C(n,\al,\Si_0,\Phi)$ is a constant independent of $\de$.
\end{lemma}

\begin{proof}
By Lemma \ref{bed-bdd} and Lemma \ref{lem-speed-upper}, we have
\begin{align*}
K^\al = -\pa_t u-\bed(u-\vp)\le C
\end{align*}
which completes the proof.
\end{proof}

What we have proved in this and the previous section is the following uniform estimate on the Gauss curvature: if $\{\Si_t\}_{0\le t< T}$ is a solution of \eqref{astde}, then its Gauss curvature satisfies
\begin{align}\label{ineq-cKC}
0<c_T\le K(z,t) \le C \qu \text{for all} \qu (z,t)\in \bS^n\ti [0,T),
\end{align}
where $c_T$ and $C$ are constants independent of $\de$. This, however, does not give sufficient controls on each principal curvature. In the next section we will obtain uniform bounds on each principal curvature, which implies the optimal regularity for the solution of \eqref{eq-main-u}.

\section{Uniform bounds on principal curvatures}\label{sec:c<la<C}
Here we establish uniform (independent of $\de$) positive lower bounds on the principal curvatures of the solution to \eqref{astde} using the bounds \eqref{ineq-cKC}. We start with two lemmas that will be used in the proof of Lemma \ref{lem-lac} below. In the following lemma and its proof, we will not use the Einstein summation convention temporarily.

\begin{lemma}[Euler type formula]\label{l-euler}
Let $\Si\su \R^{n+1}$ be a smooth, strictly convex hypersurface, and let $X:\bS^n\ra \R^{n+1}$ be an immersion such that $\Si=X(\bS^n)$ parameterized through the inverse of the Gauss map, i.e., $\nu(X(z))=z$ for any $z\in \bS^n$. Then for any $z\in \bS^n$ and $1\le i\le n$,
\begin{align*}
\fr{h_{ii}(z)}{\og_{ii}(z)} \le\fr{1}{\la_{\min}(z)},
\end{align*}
where $\og_{ij}$ is the standard metric on $\bS^n$ and $h_{ij}$ is the second fundamental form of $\Si$. 
\end{lemma}

\begin{proof}
Fix a point $z\in \bS^n$ and an orthonormal basis $\{E_1,\cdots,E_n\}$ of $T_z\Si$ such that $L(E_j)= \la_j E_j$ for $j=1,\cdots,n$, where $L$ is the Weingarten map and $\la_1,\cdots,\la_n$ are the principal curvatures of $\Si$ at $z$. Write $\oD_iX = \sum_{j}a_{ij}E_j$ with $a= (a_{ij})$, and denote by $c=\{c_{ij}\}$ the diagonal matrix $\operatorname{diag}(\la_1,\cdots,\la_n)$. Since $L(\D_iX)=\sum_{j}h_i^j(\D_jX)$, we get
\begin{align*}
\sum_{j,k}a_{ij}c_{jk}E_k=L(\D_iX) =h_i^j(\D_jX)=\sum_{j,k}h_i^ja_{jk}E_k
\end{align*}
which implies 
\begin{align}\label{eq-ac}
(ac)_{ik} = \sum_ja_{ij}c_{jk} = \sum_{l,j}h_{il}g^{lj}a_{jk}=(hg^{-1}a)_{ik},
\end{align}
where $h$ and $g$ are $n\ti n$ matrices whose $(i,j)$-components are $h_{ij}$ and $g_{ij}$, respectively, and $g^{-1}$ is the inverse matrix of $g$. 
Observing that 
\begin{align*}
g_{ij}= \inn{\D_iX}{\D_jX}=\sum_{k,l}a_{ik}\de_{kl}a_{jl}=\sum_ka_{ik}a_{jk}=(aa^T)_{ij},
\end{align*}
it follows from \eqref{eq-ac} that $h= aca^{-1}g=aca^T$.

On the other hand, from \eqref{eq-hggb} we get
\begin{align*}
\og =hg^{-1}h=(aca^T)(aa^T)^{-1}aca^T=aca^T(a^T)^{-1}a^{-1}aca^T=ac^2a^T
\end{align*}
which implies
\begin{align*}
\og_{ii} =\sum_{j,k,l}a_{ij}c_{jk}c_{kl}a_{il} \ge \la_{\min} \sum_{j,l,l}a_{ij}c_{jk}\de_{kl}a_{il}= \la_{\min} \sum_{j,k}a_{ij}c_{jk}a_{ik}= \la_{\min}h_{ii}.
\end{align*}
This completes the proof.
\end{proof}

\begin{lemma}\label{l-ev-ebed}
If $u:\bS^n\ti [0,T)\ra \R$ is a smooth solution to \eqref{astde}, then the evolution equation of the following quantity involving $\bed=\bed(u-\vp)$ is
\begin{align*}
\op e^{-\chi\bed}= e^{-\chi\bed}\left [-\chi\bed' \op(u-\vp)+(-\chi^2(\bed')^2+\chi\bed'')\norm{\oD(u-\vp)}^2_\cL\right],
\end{align*}
where $\chi$ is a given constant.
\end{lemma}

\begin{proof}
The proof follows from direct computations. In fact, we have
\begin{align*}
\pa_t e^{-\chi\bed}&=-\chi\bed'e^{-\chi\bed} \pa_t(u-\vp),\\
\oD_j e^{-\chi\bed}&= -\chi\bed' e^{-\chi\bed} \oD_j(u-\vp),\\
\oD_i\oD_j e^{-\chi\bed}&=-\chi\bed' e^{-\chi\bed} \oD_i\oD_j(u-\vp)+((\chi\bed')^2-\chi\bed'')e^{-\chi\bed}\oD_i(u-\vp)\oD_j(u-\vp)
\end{align*}
so that $\cL e^{-\chi\bed}=-\chi\bed' e^{-\chi\bed} \cL(u-\vp)+((\chi\bed')^2-\chi\bed'')e^{-\chi\bed}\norm{\oD (u-\vp)}_\cL^2$ and the conclusion follows.
\end{proof}

Now we are ready to prove the main result of this section.

\begin{lemma}\label{lem-lac}
Let $\{\Si_t\}_{0\le t< T}$ be a solution to \eqref{astde} and $\Phi$ be a shrinking obstacle. Then the principal curvatures $\la_1(\cdot,t),\cdots,\la_n(\cdot,t)$ of $\Si_t$ satisfy 
\begin{align}\label{e-la>c}
\inf_{\bS^n\ti[0,T)}\la_i \ge c_T>0,
\end{align}
where $c_T$ is a constant depending only on $n, \al, T, \Phi$, and $\Si_0$. In particular, the constant $c_T$ depends on the minimum speed of the obstacle, $\inf_{\bS^n\ti [0,T]} |\pa_t\Phi|$.
\end{lemma}

\begin{remark}
If the obstacle is stationary, i.e. $\Phi_t\equiv \Phi_0$, then there is no positive lower bounds for $\la_i$, $i=1,\ldots,n$, see \cite{LL21_CVPDE}.
\end{remark}

\begin{proof}
Take a time $T'\in (0,T)$.
To establish the lower bound \eqref{e-la>c}, we estimate an upper bound of a function 
\begin{align*}
\tilde W(z,t)= \la_{\min}^{-1}(z,t) e^{-\chi\bed(u-\vp)},
\end{align*}
where $\displaystyle\la_{\min}(z,t)=\min_{i=1,\cdots,n} \la_i(z,t)$ and $\chi$ is a constant to be determined later.
Assume that $\tilde W$ attains its maximum value over $\bS^n\ti [0,T']$ at an interior point $(z_0,t_0)$ with $t_0>0$. 
Now we choose a coordinate chart of $z_0$ such that 
\begin{align}\label{gh-normal}
\og_{ij}(z_0,t_0)=\de_{ij} \qu \text{and} \qu h_{ij}(z_0,t_0)=\la_i^{-1}(z_0,t_0)\de_{ij}
\end{align}
with $\la_1\le \cdots\le \la_n$.
We then note that a function
\begin{align*}
W(z,t)= \fr{h_{11}}{\og_{11}}e^{-\chi\bed(u-\vp)}
\end{align*}
also has the same maximum at the same point $(z_0,t_0)$ since, by Lemma \ref{l-euler}, we see that
\begin{align*}
W(z,t)&\le \tilde W(z,t) \le \tilde W(z_0,t_0),\\
W(z_0,t_0)&=\la_{1}^{-1}e^{-\chi\bed(u-\vp)}(z_0,t_0)=\tilde W(z_0,t_0).
\end{align*}
Thus we can obtain the upper bound for the function $\tilde W$ by estimating the function $W$, and
we have at the point $(z_0,t_0)$, 
\begin{align}\label{Watmax}
\pa_tW\ge 0, \qu \oD W=0, \qu \text{and}\qu \oD^2 W \le 0. 
\end{align}

Our next task is to derive an evolution equation for $W$. We first observe that
\begin{align*}
\pa _t \ov g=0 \qu\text{and}\qu \D \ov g=0,
\end{align*}
and recall the evolution equations for $h_{ij}$ from \eqref{ev-h}:
\begin{align*}
(\pa_t-\cL)h_{ij}=-\mathcal{C}_{ij}+(n\al-1)K^\al \og_{ij}-\al K^\al H h_{ij} -\oD_{ij}^2\bed -\bed \og_{ij}.
\end{align*}
From these, we have
\begin{align*}
(\pa_t-\cL)\left(\fr{h_{11}}{\og_{11}}\right)=-\fr{\mathcal{C}_{11}}{\og_{11}}+(n\al-1)K^\al  -\al K^\al H\fr{h_{11}}{\og_{11}}-\fr{\oD_{11}^2\bed}{\og_{11}} -\bed.
\end{align*}
Observe that at the point $(z_0,t_0)$ we have $n-Hh_{11}/\og_{11}=n-H/\lambda_1\le 0$, and thus 
\begin{align}\label{reactions<0}
(n\al-1) K^\al -\al K^\al H \fr{h_{11}}{\og_{11}}-\bed
&= \al K^\al \left (n-H\fr{h_{11}}{\og_{11}}\right)-K^\al -\bed<0
\end{align}
since $K^\al +\bed=-\pa_t u >-\pa_t \vp>0$ by Lemma \ref{l-uvpt}. However, the second derivative of $\bed$ produces bad terms, which makes us to consider the auxiliary function $e^{-\chi\bed}=e^{-\chi\bed(u(z,t)-\vp(z,t))}$. By Lemma \ref{l-ev-ebed}, we have that
\begin{align}\label{ev-W}
\begin{split}
(\pa_t-\cL)W
&=(\pa_t-\cL)\left(\fr{h_{11}}{\og_{11}}\right)e^{-\chi\bed}
+\fr{h_{11}}{\og_{11}}(\pa_t-\cL)\left(e^{-\chi\bed}\right)
\\
&\qu -2\inn{\oD \left(\fr{h_{11}}{\og_{11}}\right)}{\oD e^{-\chi\bed} }_\cL 
\\
&=\left(-\fr{\mathcal{C}_{11}}{\og_{11}}+(n\al-1)K^\al  -\al K^\al H\fr{h_{11}}{\og_{11}}  -\fr{\oD_{11}^2\bed}{\og_{11}} -\bed\right)e^{-\chi\bed}
\\
&\qu+\left( -\chi\bed'e^{-\chi\bed} \op(u-\vp)+(-\chi^2(\bed')^2+\chi\bed'')e^{-\chi\bed}\norm{\oD(u-\vp)}^2_\cL \right)\fr{h_{11}}{\og_{11}}
\\
&\qu -2e^{\chi\bed}\inn{\oD W}{\oD e^{-\chi\bed} }_\cL 
+2\fr{h_{11}}{\og_{11}}e^{\chi\bed}\norm{\oD e^{-\chi\bed}}^2_\cL.
\end{split}
\end{align}

We finally estimate \eqref{ev-W} at the point $(z_0,t_0)$.
Since $\norm{\oD e^{-\chi\bed}}^2_\cL=(\chi\bed')^2e^{-2\chi\bed}\norm{\oD (u-\vp)}^2_\cL$, by dividing by $e^{-\chi\bed}$ in \eqref{ev-W}, we see that \eqref{Watmax} and \eqref{reactions<0} implies at the point $(z_0,t_0)$,
\begin{align*}
0&<-\fr{\mathcal{C}_{11}}{\og_{11}}  -\fr{\oD_{11}^2\bed}{\og_{11}} 
+\left( -\chi\bed' \op(u-\vp)+(\chi^2(\bed')^2+\chi\bed'')\norm{\oD(u-\vp)}^2_\cL \right)\fr{h_{11}}{\og_{11}}.
\end{align*}
By a direct computation, we have
\begin{align*}
\oD^2_{11}\bed= \bed' \oD^2_{11}(u-\vp)+\bed'' |\oD_1(u-\vp)|^2.
\end{align*}
We then deduce from \eqref{gh-normal} that
\begin{align}\label{i-atWmax1}
\begin{split}
0&<-\mathcal{C}_{11}  -\bed'' |\oD_1(u-\vp)|^2 +(\chi^2(\bed')^2+\chi\bed'')\la_1^{-1}\norm{\oD(u-\vp)}^2_\cL
\\
&\qu -\bed' \oD^2_{11}(u-\vp)-\chi\bed' \la_1^{-1}\op(u-\vp).
\end{split}
\end{align}
On the other hand, using \eqref{gh-normal} again to simplify the quantity $\mathcal{C}_{11}$, we see
\begin{align*}
\mathcal{C}_{11}&= \al^2 K^\al \left(\sum_{k=1}^n\la_k\oD_1h_{kk}\right)^2+ \al K^\al \sum_{k,l=1}^n\la_k\la_l (\oD_1h_{kl})^2
\\
&\ge\al K^\al \sum_{i=1}^n\la_i\la_1 (\oD_1h_{i1})^2=\al K^\al (\chi\bed')^2\sum_{i=1}^n\fr{\la_i}{\la_1}|\oD_i(u-\vp)|^2,
\end{align*}
where we have used $\la_1 \oD h_{11}=\chi\bed'\oD(u-\vp)$ which follows from $\oD W=0$ at $(z_0,t_0)$ and the Codazzi equation \eqref{e-codazzi} in Lemma \ref{l-codazzi}.
Thus we obtain
\begin{align*}
\fr{(\chi^2(\bed')^2+\chi\bed'')}{\la_1}\norm{\oD(u-\vp)}^2_\cL&\le\mathcal{C}_{11}+\al K^\al \chi \bed''|\oD_1(u-\vp)|^2
\end{align*}
since $\norm{\cdot}^2_\cL=\al K^\al \sum_{i=1}^n\la_i|\oD_{i}\cdot|^2$. Moreover, using Corollary \ref{c-K>c}, we can take $\chi=\chi(\al,T,\Phi)>0$ satisfying $\al K^\al \chi\ge 1$ which implies
\begin{align*}
(\al K^\al \chi -1 )\bed'' |\oD_1(u-\vp)|^2\le 0
\end{align*}
since $\bed''\le 0$. Combining these facts together, therefore, the inequality \eqref{i-atWmax1} becomes, after dividing $\bed'$,
\begin{align}\label{ineq-almost final}
0<-\oD^2_{11}(u-\vp)-\chi \la_1^{-1}\op(u-\vp).
\end{align}

To finish the proof, we observe that at the point $(z_0,t_0)$, 
\begin{align*}
\oD^2_{ii}u= h_{ii}-u=\fr{1}{\la_i}-u \qu\text{and}\qu \oD^2_{ii}\vp= h_{ii}^\vp-\vp=\fr{1}{\mu_i}-\vp,
\end{align*}
where $\{\mu_i\}_{i=1,\cdots,n}$ is principal curvatures of $\Phi$ and $h^\vp$ is the second fundamental form of $\Phi$. Using the evolution equation of $u$ in Lemma \eqref{evs}, we have
\begin{align*}
  \op (u-\vp)&=\al K^\al Hu-(n\al +1)K^\al -\bed -\vp'+\al K^\al \sum_{i=1}^n\la_i\left(\fr{1}{\mu_i}-\vp\right)\\
  &=\al K^\al H(u-\vp)-(n\al +1)K^\al -\bed -\vp'+\al K^\al \sum_{i=1}^n\fr{\la_i}{\mu_i}
  \\
  &\ge -(n\al +1)K^\al +\al K^\al \sum_{i=1}^n\fr{\la_i}{\mu_i}
\end{align*}
since $H=\la_1+\cdots+\la_n>0$, $\bed\le 0$, $\vp'<0$, and $u-\vp>0$. 
Hence, \eqref{ineq-almost final} becomes 
\begin{align*}
0<-\fr{1}{\la_1}+\fr{1}{\mu_{\min}}+(u-\vp)+\fr{K^\al \chi}{\la_1}\left(n\al+1 -\al  \fr{\la_n}{\mu_{\max}}\right). 
\end{align*}
Thus we conclude that
\begin{align*}
-\fr{1}{\la_1}+\fr{1}{\mu_{\min}}+(u-\vp)>0 \qu\text{or}\qu n\al+1 -\al  \fr{\la_n}{\mu_{\max}}>0,
\end{align*}
which is equivalent to
\begin{align}\label{ineq-la1orlan}
\fr{1}{\la_1}<\fr{1}{\mu_{\min}}+u-\vp \qu\text{or}\qu \la_n<\left(n+\fr{1}{\al}\right)\mu_{\max}.
\end{align}
For the latter inequality, we relate the largest eigenvalue $\la_n$ with the smallest eigenvalue $\la_1$ by using Corollary \ref{c-K>c}. In fact, we have 
\begin{align}\label{ineq-la1lan}
\fr{1}{\la_1} = \fr{\la_2\cdots \la_n}{K} \le \fr{\la_n^{n-1}}{c},
\end{align}
where $c$ is the constant in Corollary \ref{c-K>c}. From \eqref{ineq-la1orlan} and \eqref{ineq-la1lan}, there exists a positive constant $C=C(n,\al,T,\Phi,\Si_0)$ such that
\begin{align*}
\fr{1}{\la_1} \le C.
\end{align*}
Using this and Lemma \ref{bed-bdd}, we finally conclude that
\begin{align}\label{ineq-lambda-last}
\fr{1}{\la_{\min}}\le\max_{\bS^n\ti[0,T']}\tilde W=\tilde W(z_0,t_0)=W(z_0,t_0)=\fr{e^{-\chi\bed(u-\vp)}}{\la_1}\le Ce^{\chi C_0}.
\end{align}

Since $T'$ is an arbitrary number in $(0,T)$ and the upper bound in \eqref{ineq-lambda-last} does not depend on $T'$, we obtain the conclusion by taking $T'\ra T$. 
\end{proof}

The lemma above automatically gives uniform upper bounds on principal curvatures. Indeed, the largest principal eigenvalue satisfies
\begin{align*}
\la_n = \fr{K}{\la_1\cdots \la_{n-1}}\le \fr{K}{\la_1^{n-1}}
\end{align*}
which is bounded by Lemma \ref{lem-KC} and Lemma \ref{lem-lac}.
In summary, all principal curvatures of the solution $\{\Si_t\}$ to \eqref{astde} over $\bS^n\ti [0,T)$ satisfy the following uniform estimates:
\begin{align}\label{lambda-bounds}
0<c_T\le \la_i(z,t) \le C_T \qu \text{for all}\qu (z,t)\in \bS^n\ti [0,T), 1\le i \le n,
\end{align}
where $c_T$ and $C_T$ are constants independent of $\de$.

\section{Proof of Theorem \ref{t-main}}\label{sec:pfThm1}
The proof of Theorem \ref{t-main} consists of two parts. 
First, we provide an existence result for the long time solution that have the optimal $C^{1,1}$ regularity.
In this part we will use the uniform estimates obtained in the previous sections.
Second, we show that the motion of the solution is identically equal to that of the obstacle after some time.

\begin{proof}[proof of Theorem \ref{t-main}]
From Lemma \ref{lem-locE}, approximate solutions to \eqref{astde} exist at least for short time. Let $T>0$ be the maximal time for which the solutions exist. We claim that $T=\infty$. If not, we apply Lemma \ref{lem-KC} and Lemma \ref{lem-lac} to the solutions so that \eqref{lambda-bounds} holds over $\bS^n\ti [0,T)$. Then the linearized operator $\cL$ is uniformly parabolic, i.e., 
\begin{align*}
0<\fr{1}{C}|\xi|^2\le \al K^\al b^{ij}\xi_i\xi_j \le C |\xi|^2 \qu \text{on } \bS^n\ti [0,T)
\end{align*}
for all $\xi=(\xi_i)\in \R^n$, where $C=C(n,\al,\Si_0,\Phi,T)$ is a positive constant. By applying the standard argument in parabolic theory \cite{Lieberman96_book}, the solutions exist beyond $T$, which is contradict to the maximality of $T$. Hence, the solutions to \eqref{astde} exist for all time.

Next we prove the uniqueness of solutions to \eqref{eq-main-u}. Let $u_1$ and $u_2$ be two viscosity solutions of \eqref{eq-main-u} with the same initial data. Assume that $u_1<u_2$ at some point $(z_0,t_0)$. Since $u_1=u_2$ at the initial time, we have $t_0>0$. Observing that $\vp\le u_1 < u_2$ in the set
\begin{align*}
G= \{\bS^n \ti (0,t_0]: u_1<u_2\},
\end{align*}
we have $-\pa_t u_2 = K_2^\al$, where $K_2$ is the Gauss curvature of the hypersurface produced by $u_2$. On the other hand, $u_1$ satisfies $-\pa_tu_1 \le K_1^\al$, where $K_1$ is the Gauss curvature with respect to $u_1$, similarly. In other words, $u_1$ is a supersolution and $u_2$ is a subsolution of the same equation $-\pa_t u = K^\al $. Since $u_1=u_2$ on the parabolic boundary of $G$, it follows from the comparison principle that $u_2\le u_1$ in $G$ which is a contradiction. Therefore, $u_1\equiv u_2$, and the solution to \eqref{eq-main-u} is unique.

Fix a time $T\in(0,\infty)$, and let $u^\de$ be a solution of \eqref{astde} over $\bS^n\ti [0,T)$. By the estimates on the principal curvatures \eqref{lambda-bounds}, the family $\{u^\de\}_{\de>0}$ of solutions is uniformly bounded in $C^{1,1}$ so that there exists a function $u\in C^{1,1}$ such that $u^\de $ converges over a subsequence $\de=\de_j\ra 0$ to $u$ in $C^{1,\beta}$ for any $0<\beta<1$. Moreover, it follows from the uniform boundedness of the penalty term $\bed$ obtained in Lemma \ref{bed-bdd} that $u$ satisfies $u\ge \vp$. Since $T$ is arbitrary, we have the viscosity solution $u$ of \eqref{eq-main-u} that exists for all time with controlled principal curvatures
\begin{align}\label{ineq-ClaC}
\fr{1}{C_T}\le \la_i\le C_T, \qu \text{for } i=1,\cdots,n,
\end{align}
in $\bS^n\ti [0,T)$, where $C_T$ is a constant depending on $T$. We will remove the dependence of $T$ in the constant $C_T$ after analyzing the long time behavior of the solution. 

Now we prove that the solution coincides with the shrinking obstacle after some time $T_*$. Fix a point $(z_1,t_1)\in \bS^n\ti [0,\infty)$ and consider the point $\Phi(z_1,t_1)$ on the obstacle $\Phi(\cdot,t_1)$. Since the principal curvatures of the shrinking obstacle satisfy
\begin{align}\label{ineq-mumu}
\mu_i(\cdot,t)\ge \mu_i(\cdot,0)>0,\qu\text{for all }i=1,\cdots,n,
\end{align}
we can take a ball touching at the point $\Phi(z_1,t_1)$ and containing the obstacle $\Phi(\cdot,t_1)$ at $t_1$. Denoting by $a(z_1,t_1)$ and $r(z_1,t_1)$ the center and the radius of the ball, respectively, we write the ball as
\begin{align}\label{a+B}
a(z_1,t_1)+B_{r(z_1,t_1)}.
\end{align}
By \eqref{ineq-mumu} and $\conv\Phi(\cdot,t)\su \conv\Phi(\cdot,0)$ for all $t\ge0$, we may assume that $r(z_1,t_1)\le \rho$ for some constant $\rho$ independent of $(z_1,t_1)$.

We will construct a barrier for the point $\Phi(z_1,t_1)$ with 
\begin{align}\label{ineq-tur}
t_1> \fr{1}{n\al+1} (\norm{u_0}_\infty+\rho)^{n\al+1}.
\end{align} 
Recall that the solution of the $\al$-Gauss curvature flow of the ball centered at $a_0$ with radius $R_0$ at time $t$ is given by $a_0+ \pa B_{R(t)}$ where $R(t)=(R_0^{n\al +1}-(n\al+1)t))^{1/(n\al+1)}$. Thus the $\al$-Gauss curvature flow whose shape at $t_1$ is equal to the boundary of \eqref{a+B} can be written as
\begin{align}\label{a+pB}
a(z_1,t_1)+\pa B_{R(t)}=\{x\in \R^{n+1}:|x-a(z_1,t_1)|=R(t)\}
\end{align}
where $R(t)= (r(z_1,t_1)^{n\al+1}+(n\al+1)(t_1-t))^{1/(n\al+1)}$. Then the support function $u_p$ of the ball \eqref{a+pB} is 
\begin{align*}
u_p(z,t) = \inn{a(z_1,t_1)}{z}+R(t) 
\end{align*}
Note that $u_p$ is concave in $t$ so that $u_p-\vp$ is also concave in $t$, and $u_p-\vp$ is nonnegative in $\bS^n\ti \{t_1\}$ since the ball $a(z_1,t_1)+B_{R(z_1,t_1)}$ contains the obstacle $\Phi(\cdot,t_1)$. 
Furthermore, from the fact that the origin is contained in the obstacle $\Phi(\cdot,t_1)$ which is again contained in the ball $a(z_1,t_1)+B_{R(t)}$, we have 
\begin{align*}
|a(z_1,t_1)|=|0-a(z_1,t_1)|\le R(t_1)=r(z_1,t_1).
\end{align*}
This gives that at $t=0$, 
\begin{align*}
u_p(z,0)-\vp(z,0)&=R(0)+\inn{a(z_1,t_1)}{z}-\vp_0(z)
\\
&\ge (r(z_1,t_1)^{n+1}+(n\al+1)t_1)^{1/(n\al+1)}-|a(z_1,t_1)| -\norm{\vp_0}_{\infty}
\\
&\ge ((n\al+1)t_1)^{1/(n\al+1)}-\rho-\norm{u_0}_\infty ,
\end{align*}
where we have used $|a(z_1,t_1)|\le r(z_1,t_1)\le \rho$ and $\vp_0\le u_0$.
Then it follows from \eqref{ineq-tur} that $u_p-\vp> 0$ at $t=0$. This, together with the concavity of $u_p-\vp$ and $u_p-\vp\ge0$ at $t=t_1$, concludes that
$u_p> \vp$ in $\bS^n\ti [0,t_1)$, and therefore, the solution $u_p$ of the $\al$-Gauss curvature flow is also the solution of \eqref{eq-main-u}, the $\al$-Gauss curvature flow with the shrinking obstacle. Moreover, at the initial time $t=0$, from \eqref{ineq-tur}, we have
\begin{align*}
u_p(z,0)-u_0\ge ((n\al+1)t_1)^{1/(n\al+1)}-\rho-\norm{u_0}_\infty>0.
\end{align*}

In other words, both $u_p:\bS^n\ti[0,t_1)\ra\R$ and $u:\bS^n\ti[0,t_1)\ra\R$ are the solution to the $\al$-Gauss curvature flow with the shrinking obstacle, and the initial data of $u_p$ encloses the initial data $u_0$, which imply $u_p\ge u$ in $\bS^n\ti[0,t_1)$ by the comparison principle. Since $u_p(z_1,t_1)=\vp(z_1,t_1)$ and $u_p\ge u\ge\vp$, we conclude that $u(z_1,t_1)=\vp(z_1,t_1)$.
Hence, setting $T_*=(\norm{u_0}_\infty+\rho)^{n\al+1}/(n\al+1)$ yields $u\equiv \vp$ in $\bS^n\ti [T_*,\infty)$ since $(z_1,t_1)$ was an arbitrary point in $ \bS^n\ti (T_*,\infty)$.

It remains to prove that the principal curvatures $\la_1,\cdots,\la_n$ are globally bounded from above and below in $\bS^n\ti[0,\infty)$. Taking $T=T_*$ in \eqref{ineq-ClaC}, we get $C_{T_*}^{-1}\le \la_i\le C_{T_*}$ on the time interval $[0,T_*)$. On another time interval $[T_*,\infty)$, since $u$ is identically equal to $\vp$ and the principal curvatures $\mu_1,\cdots,\mu_n$ of the obstacle is globally bounded from above and below in $\bS^n\ti[0,\infty)$, we can obtain the desired bounds. This completes the proof. 
\end{proof}

\section{Proof of Theorem \ref{t-main-2}}\label{sec:pfThm2}
To prove that the free boundary is locally $C^1$ graph, we need several ingredients: non-degeneracy (Lemma \ref{lem-non-deg}), classification of blowups \eqref{class-v0}, continuity of speed \eqref{eq:vt=0}, and directional monotonicity (Lemma \ref{lem:dir-mono} and \eqref{eq-vr-dir}).

Let $w$ be a solution to the local graph representation of the $\al$-Gauss curvature flow with a shrinking obstacle, \eqref{eq-main-w}. Set $v= \phi -w$. Then $v:Q_1\ra\R$ satisfies
\begin{equation}\label{eq-local-v}
\begin{alignedat}{3}
F(D^2v,Dv,x,t)-\fr{\pa}{\pa t} v &= f(Dv,x,t) \qu &&\text{in} \qu \Om(v)=\{v>0\},
\\
v &\ge 0 \qu &&\text{in} \qu Q_1,
\\
|D^2v|+|\pa_t v| &\le C \qu &&\text{in} \qu Q_1, 
\end{alignedat}
\end{equation}
where 
\begin{align*}
F(M,p,x,t) &=\fr{\det D^2\phi(x,t)^\al - \det(D^2\phi(x,t)-M)^\al}{(1+|D\phi(x,t)-p|^2)^{\frac{(n+2)\al-1}{2}}},
\\
f(p,x,t) & = -\fr{\pa}{\pa t} \phi(x,t)+\fr{\det D^2\phi(x,t)^\al }{(1+|D\phi(x,t)-p|^2)^{\frac{(n+2)\al-1}{2}}}.
\end{align*}

Observe that $F(0,Dv,x,t)=0$, $f \ge c>0$ for some constant $c$, and $F(\cdot,Dv)$ is convex if $\al\le\fr{1}{n}$. Moreover, the operator $F$ is uniformly elliptic since the eigenvalues of $D^2w$ are bounded above and below by positive constants; there exists a constant $\theta>0$ such that $\theta^{-1}|\xi|^2\le F^{ij}\xi_i\xi_j\le \theta|\xi|^2$, where
\begin{align*}
F^{ij}:=\fr{\pa F}{\pa M_{ij}}(D^2v,Dv,x,t)=\fr{\al \det(D^2w)^\al [(D^2w)^{-1}]^{ij}}{(1+|D\phi-Dv|^2)^{\frac{(n+2)\al-1}{2}}}.
\end{align*}

Let us fix some notations. For a point $X=(x,t)$, let $Q_r(X)= B_r(x)\ti(t-r^2,t)$. We denote by $\pa_pQ_r(X)$ the parabolic boundary of $Q_r(X)$.

In Theorem \ref{t-main}, we proved the optimal curvature bounds for the solution, which implies the optimal $C^{1,1}$ regularity of the solution. In particular, the solution has at most quadratic growth especially near the free boundary. The following lemma says the quadratic growth is actually achieved near the free boundary. 

\begin{lemma}[non-degeneracy]\label{lem-non-deg}
Let $v$ be a solution of \eqref{eq-local-v} with $\Om=\Om(v)$. Then for any $X_0=(x_0,t_0)\in \ov\Om\cap Q_{1/2}$, we have the inequality
\begin{align}\label{ineq-nondeg}
\sup_{\pa_p Q_r(X_0)} v \ge v(X_0)+\fr{cr^2}{2n\theta+1} \qu \text{for all} \qu 0<r<1/4.
\end{align}
\end{lemma}

\begin{proof}
By an approximation argument, it suffices to prove \eqref{ineq-nondeg} for $X_0\in \Om\cap Q_{1/2}$. For a point $X_0=(x_0,t_0)\in \Om\cap Q_{1/2}$, we define a function $\ov v:Q_1\ra \R$ by 
\begin{align*}
\ov v(x,t) = v(x,t)- \fr{c(|x-x_0|^2-(t-t_0))}{2n\theta+1}.
\end{align*}
We claim that $\ov v$ satisfies $F(D^2\ov v,D v,x,t)-\pa_t\ov v\ge0$ in $\Om\cap Q_1$. If the claim holds, then it follows from $\ov v(X_0)=v(X_0)\ge 0$ and $\ov v=- \fr{c(|x-x_0|^2-(t-t_0))}{2n\theta+1}<0$ on $\pa\Om$ that 
\begin{align*}
0\le \sup_{\Om\cap Q_r(X_0)} \ov v\le \sup_{\pa_pQ_r(X_0)} \ov v\le \sup_{\pa_pQ_r(X_0)}  v-\fr{cr^2}{2n\theta+1},
\end{align*}
by the maximum principle, which proves \eqref{ineq-nondeg}.

We now prove the claim. Note that $D^2\ov v =D^2v -\fr{2c}{2n\theta+1}I$, where $I$ denotes the $n\ti n$ identity matrix. Then by the uniformly ellipticity, we have 
\begin{align*}
F(D^2\ov v,D v,x,t)-\pa_t\ov v&=F(D^2v-\fr{2c}{2n\theta+1}I,D v,x,t)-\pa_tv -\fr{c}{2n\theta+1}
\\
&\ge F(D^2v,D v,x,t)-\fr{2n\theta c}{2n\theta+1}-\pa_tv -\fr{c}{2n\theta+1}
\\
&=f(Dv,x,t)-c \ge 0 \qu \text{in }\Om.
\end{align*}
Therefore, the claim holds.
\end{proof}

Take a free boundary point $X_0=(x_0,t_0) \in \pa\Om\cap Q_{1/8}$ and consider the rescaled function $v_r$ ($r>0$) of $v$ around $X_0$ defined by 
\begin{align*}
v_r(y,s)=\fr{v(x_0+ry,t_0+r^2s)-v(x_0,t_0)}{r^2}, \qu (y,s)\in Q_{\tfrac{1}{2r}}.
\end{align*}
By Theorem \ref{t-main} and the scaling properties $D^2_yv_r(y,s) = D^2_xv(x_0+ry,t_0+r^2s)$ and $ \pa_sv_r(y,s) = \pa_tv(x_0+ry,t_0+r^2s)$, the rescaled functions $\{v_r\}_{r>0}$ have uniform $C^{1,1}$ estimates. Then we can extract a converging subsequence $v_{r_j}\ra v_0$ in $C_{\text{loc}}^{1,\ga}(\R^n\ti \R)$ for any $0<\ga<1$, where $v_0\in C^{1,1}_{\text{loc}}(\R^n\ti\R)$.

By the standard argument for blowups (see Lemma 18 in \cite{IM16_AIHPANL} or Proposition 3.17 in \cite{PSU12_Book}), together with the non-degeneracy (Lemma \ref{lem-non-deg}), $v_0$ satisfies 
\begin{align}\label{eq-v0}
F(D^2v_0(y,s),0,x_0,t_0)-\pa_s v_0(y,s)=f(0,x_0,t_0)\qu \text{in} \qu \Om(v_0).
\end{align}
Indeed, it follows from the equation
\begin{align*}
F(D^2v_{r_j},r_jDv_{r_j},x_0+r_jy,t_0+r_j^2s)-\pa_sv_{r_j}
&=f(r_jDv_{r_j},x_0+r_jy,t_0+r_j^2s)
\end{align*}
in $\Om(v_{r_j})$ by taking $j\ra \infty$, where we used again the scaling properties $D_yv_r=D_xv/r$, $D^2_yv_r=D^2v$, and $D_sv_r= D_tv$.

Recall that $F$ is a convex operator in $D^2v$ variable since $\al\le 1/n$, that $F(0,0,x_0,t_0)=0$, and that $F(\cdot,0,x_0,t_0)$ is uniformly elliptic. By the work of Figalli and Shahgholian \cite[Proposition 3.2]{FS15_AMPA} on the classification of global solutions to \eqref{eq-v0}, it follows from the uniform thickness condition in Theorem \ref{t-main-2} that, after a rotation,
\begin{align}\label{class-v0}
v_0(y,s)= \tfrac{\ga}{2}[(x_1)_+]^2,
\end{align}
where $\ga\in (1/\theta,\theta)$ such that $F(\ga e_1\otimes e_1,0,x_0,t_0) = f(0,x_0,t_0)$. In particular, $v_0$ is time-independent. 

We now claim the continuity of the speed, 
\begin{align}\label{eq:vt=0}
\lim_{X \ra \pa\Om}\pa_tv(X)=0,
\end{align}
under the uniform thickness assumption.
We prove the claim by contradiction. 
If there exists a sequence $X_j\in \Om(v)$ such that $X_j \ra \pa\Om(v)$ and $|\pa_t v(X_j)|\ge \e $ for some $\e>0$, we define $d_j=\text{dist}(X_j,\pa\Om(v))$ and consider 
\begin{align*}
v_{d_j}(y,s)=\fr{v(x_j+d_jy,t_j+d_j^2s)-v(x_j,t_j)}{d_j^2}.
\end{align*}
Then it can be verified that $v_{d_j}$ converges to a global solution $\tilde v_0$ of \eqref{eq-v0} with $|\pa_s \tilde v(0)|>0$, which contradicts the fact that the global solution $\tilde v_0$ is time-independent.

Our next lemma is the key ingredient of the directional monotonicity for solutions to 
\begin{align}\label{eq-vr}
F(D^2v,rDv,x_0+ry,t_0+r^2s)-\pa_sv
&=f(rDv,x_0+ry,t_0+r^2s).
\end{align}
Note that the rescaled function $v_r$ solves \eqref{eq-vr} in $\Om(v_r)$.
For simplicity, we assume for a moment $(x_0,t_0)=(0,0)$.
\begin{lemma}\label{lem:dir-mono}
Assume that $v$ satisfies \eqref{eq-vr} in $\Om(v)$. For a number $\kappa\in(0,1)$, let $e=(e_x,e_t)\in \bS^n$ be any space-time direction satisfying $e\cdot (e_1,0) \ge \kappa>0$.
If $C_\kappa\pa_e v-v\ge -\e_0$ in $Q_1$ for some constants $C_\kappa$ and $\e_0$, then $C_\kappa \pa_e v-v\ge 0$ in $Q_{1/2}$ provided that $\e_0<\fr{c}{32(2n\theta+1)}$ and $0<r\le \fr{c}{2C}$, where $C$ is a constant depending only on $\norm{f}_{C^1},\norm{F}_{C^1},\norm{v}_{C^1}$ and $\kappa$.
\end{lemma}

\begin{proof}
By convexity of $F(\cdot,rDv,ry,r^2s)$ and $F(0,rDv,ry,r^2s)=0$, we have 
\begin{align*}
F(D^2v,rDv,ry,r^2s)-F^{ij}(D^2v,rDv,ry,r^2s)\pa_{ij}v\le F(0,rDv,ry,r^2s)=0.
\end{align*}
Thus we obtain $F^{ij}\pa_{ij}v-\pa_sv\ge F-\pa_sv=f\ge c>0$. On the other hand, by differentiating \eqref{eq-vr} with respect to a direction $e\in \bS^n$, we get
\begin{align*}
Lv_e :=(F^{ij}\pa_{ij}-\pa_s+r(F_{p_i}-f_{p_i}) \pa_i)v_e=r(f_x-F_x)\cdot e_x+r^2(f_t-F_t)e_t,
\end{align*}
where $\pa_{i}=\tfrac{\pa}{\pa y_i}$ and $v_e=e_x\cdot D_yv+e_t\pa_s v$. 
Since $f_x$, $F_x$, $f_t$, and $F_t$ are bounded, the right hand side is as small as we want, provided $r>0$ is sufficiently small.

Assume by contradiction that there is $(y_0,s_0)\in Q_{\fr{1}{2}}\cap \Om(v)$ such that $C_\kappa\pa_e v(y_0,s_0)-v(y_0,s_0)<0$. If we consider
\begin{align*}
\ov v(x,t) = C_\kappa \pa_e v-v+ \fr{c(|y-y_0|^2-(s-s_0))}{2(2n\theta+1)},
\end{align*}
where $c$ is the lower bound of $f$ and $\theta$ is the ellipticity constant, then 
\begin{align*}
L\ov v&\le C_\kappa(r(f_x-F_x)\cdot e_x+r^2(f_t-F_t)e_t)
\\
&\qu-f-r((F_{p_i}-f_{p_i})\pa_iv+\fr{c}{2}+r(F_{p_i}-f_{p_i}) (y-y_0)_i
\\
&\le rC(\norm{f}_{C^1},\norm{F}_{C^1},\norm{v}_{C^1},\kappa)-\fr{c}{2}.
\end{align*}
Hence, for sufficiently small $r$, we have $L\ov v\le 0$ in $Q_{1/4}(y_0,s_0)\cap \Om(v)$. By the minimum principle, the minimum is achieved on the boundary of $Q_{1/4}(y_0,s_0)\cap \Om(v)$. However, since $\ov v>0$ on $\Om(v)$ and $\ov v(y_0,s_0)<0$, the function $\ov v$ attains its minimum on $\pa_p Q_{1/4}(y_0,s_0)$. Therefore, we conclude that
\begin{align*}
0>\min_{\pa_p(Q_{1/4}(y_0,s_0)\cap\Om(v))}\ov v = -\e_0 +\fr{c}{32(2n\theta+1)},
\end{align*}
which is a contradiction. 
\end{proof}

To finish the proof of Theorem \ref{t-main-2}, we take any number $\kappa\in (0,1)$ and consider any direction $e=(e_x,e_t)\in \bS^n\subset R^{n+1}$ such that $e\cdot (e_1,0)=e_x\cdot e_1 \ge \kappa>0$. Recall that the limit $v_0$ of $v_{r_j}$ has the form $\fr{\ga}{2}[(x_1)_+]^2$. Then for the constant $C_\kappa= 2/\kappa$, we have $C_\kappa\pa_{e} v_0-v_0\ge 0$ in $Q_1$. Using the $C^1$ convergence $v_{r_j}$ in the space variable and the continuity of the speed \eqref{eq:vt=0}, we induce $C_\kappa\pa_{e} v_{r_j}-v_{r_j}\ge -\e_0$ in $Q_1$ for sufficiently large $j\ge j_0$, where $\e_0$ is the constant given in Lemma \ref{lem:dir-mono}. Then by Lemma \ref{lem:dir-mono}, we obtain the improved inequality 
\begin{align}\label{eq-vr-dir}
C_\kappa\pa_{e} v_{r_j}-v_{r_j}\ge 0 \qu\text{in} \qu Q_{1/2}
\end{align}
for $j\ge j_0$. Since $v_{r_j} \ge0$, this implies $\pa_ev_{r_j}\ge0$ in $Q_{1/2}$. Scaling back to $v$, we conclude that there exists $r=r(\kappa)>0$ such that $\pa_e v\ge0$ in $Q_r(X_0)$ for all directions $e=(e_x,e_t)\in \bS^n$ satisfying $e\cdot (e_1,0)\ge \kappa>0$. This together with a simple compactness argument implies that the free boundary $\Ga(v)$ is $\kappa$-Lipschitz graph for any $\kappa\in (0,1)$. Then the $C^1$ graphness of the free boundary $\Ga(v)$ follows by the standard argument, see \cite{PSU12_Book} for instance.

\vspace{0.5cm}

\textbf{Acknowledgements.} 
We are grateful to the referee for helpful comments.
Ki-Ahm Lee was supported by NRF grant NRF-2020R1A2C1A01006256 funded by the Korean government (MSIP). 
Taehun Lee was supported by the NRF grant RS-2023-00211258 and KIAS Individual Grant MG079501.
\vspace{0.2cm}

\textbf{Conflict of interest.} On behalf of all authors, the corresponding author states that there is no conflict of interest.


\end{document}